\documentclass[12pt,a4paper]{amsart}
\usepackage{amsmath, amsthm, amscd, amsfonts}

\newtheorem{theorem}{Theorem}[section]

\newtheorem{proposition}[theorem]{Proposition}
\newtheorem{corollary}[theorem]{Corollary}
\theoremstyle{definition}

\numberwithin{equation}{section}

\begin{document}
\title[Toricness of Binomial Edge Ideals]{Toricness of Binomial Edge Ideals}
\author[  Mahdis Saeedi, Farhad Rahmati, Seyyede Masoome Seyyedi]{  Mahdis Saeedi, Farhad Rahmati, Seyyede Masoome Seyyedi}

\address{ Faculty of  Mathematics and Computer
Science, Amirkabir University of Technology, P. O. Box 15875-4413 ,
Tehran, Iran.} \email{mahdis.saeedi@aut.ac.ir}

\address{ Faculty of  Mathematics and Computer
Science, Amirkabir University of Technology, P. O. Box 15875-4413 ,
Tehran, Iran.}

\email{frahmati@aut.ac.ir}

\address{ Faculty of  Mathematics and Computer
Science, Amirkabir University of Technology, P. O. Box 15875-4413 ,
Tehran, Iran.}

\email{mseyyedi@aut.ac.ir}



\begin{abstract}
Let $G$ be a finite simple graph. In this paper we will show that the binomial edge ideal of $G$, ‎$ J‎_{G}‎ $‎ is toric if and only if each connected component of ‎$ G $‎  is complete and in this case it is the sum of toric ideal associated to bipartite complete graphs.

\end{abstract} \maketitle
\section{\textbf{Introduction}}
Let ‎$ G $ be a simple graph on the vertex set ‎$ [n]=‎\left\lbrace ‎1,\cdots,n ‎\right\rbrace ‎‎ $‎‎ and ‎$ K $ a field. Herzog, Hibi and Hreinsdotir [6] and Ohtani independently in [7] introduced binomial edge ideal ‎$ J‎_{G}‎ $ in ‎$ R=K[x‎_{1},\cdots ,x‎_{n}‎,y‎_{1},\cdots,y‎_{n}‎‎‎‎ ] $‎‎‎ attached to ‎$ G $ and studied their algebraic properties. They proved that ‎‎$ J‎_{G}‎ $ is a radical ideal and answered the question of when  ‎$ J‎_{G}‎ $ is a prime ideal.

%By proposition []%
The following conditions are equivalent [7]
%\begin{itemize}
%\item

i) ‎$ G $ is complete around all vertices of‎ $ G $.‎‎
%\item 

ii) Each connected component of ‎$ G $ is a complete graph.
% \item

iii) ‎$ J‎_{G}‎ $ is a prime ideal.
% \end{itemize}

 In this paper we answer to the question of when ‎ ‎$ J‎_{G}‎ $  is a lattice ideal and we will show that [theorem 3.2]  ‎$ J‎_{G}‎ $ is a lattice ideal iff ‎$ J‎_{G}= ‎\left\langle  I‎_{G‎_{1}‎}, \cdots,I‎_{G‎_{r}‎}‎\right\rangle ‎‎‎‎‎ $‎ where ‎$ I‎_{G‎_{i}‎}‎ $‎ is the toric ideal associated to a complete bipartite graph ‎$ K‎_{2,n}‎ $ equivalently  ‎$ J‎_{G}‎ $ is a prime ideal. ‎

 \section{\textbf{Preliminaries}}

Let $G$ be a simple graph and
$S=K[x_{1},\cdots,x_{n},y_{1},\cdots,y_{n}]$. set $f_{ij} =
x_{i}y_{j} -  x_{j}y_{i}.$
The binomial edge ideal $J_{G} \subset S$ of $G$ is the
ideal generated by the binomials $f_{ij} = x_{i}y_{j} -  x_{j}y_{i}$
such that $\{i, j\}$ is an edge of $G$.

We can see easily by definitions that:
\begin{proposition} $~~~~~~~~~~~~~~~~~~~~~~~$ 

~~~~~~~~~~~~~ a)Suppose that $G$ has an isolated vertex $1$, and $G'$ is the
restriction of $G$ to the vertex set $[n]\setminus
\{1\}$,$S'=K[x_{1},\cdots,x_{n-1},y_{1},\cdots,y_{n-1}]$ and
$J_{G'}$ is an ideal in $S'$, then
 $J_{G}=‎\langle‎ J_{G'}‎\rangle‎$.

b)If $G‎_{1}‎,G‎_{2}‎$ are two graphs on the vertex set $[n]$ then
$J_{G‎_{1}‎}=J_{G‎_{2}‎}$ if and only if $G‎_{1}‎=G‎_{2}‎$.

c)If $G‎_{1}‎,G‎_{2}‎$ are two graphs on the vertex set $[n]$ then $$J_{G‎_{1}‎\cup
G‎_{2}‎}= J_{G‎_{1}‎}+J_{G‎_{2}‎}.$$

d)If $G'$ is complement of $G$ with respect to $K_{n}$ then
$$J_{G}+J_{G'}= J_{K_{n}}.$$

e)Let ‎$ G‎_{1}, \cdots,G‎_{r}‎‎ $be the connected component of ‎$ G $‎ ‎then ‎$ ‎ J‎_{G}=‎\left\langle J‎_{G‎_{1}‎}, \cdots, J‎_{G‎_{r}‎}‎\right\rangle‎.‎‎‎ $‎‎

\end{proposition}

Let
$E(G)=\{e_{1},\cdots,e_{q}\}$ and 
$f_{e}=x_{i}x_{j}$  where  $e=\{i,j\}\in E(G)$ and $F=\{
f_{e_{1}},\cdots,f_{e_{q}}\}$. 
 Consider the  graded homomorphism of K$-$algebras:
$$ \varphi : B=K[t_{1},\cdots,t_{q}]\longrightarrow
K[F],$$ induced by $\varphi(t_{k})=f_{e_{k}}$ where ‎$ K[F] $‎ is the sub algebra of ‎$ K[X] $‎ generated by ‎$‎\left\lbrace ‎ f_{e_{1}},\cdots,f_{e_{q}}‎\right\rbrace‎. $‎
The kernel of $\varphi$,$I_{G}$, is the toric ideal of $K[F]$
with respect to $f_{e_{1}},\cdots,f_{e_{q}}$, named the toric ideal
associated to $G$.

Let $M$  be the ‎$ n‎\times‎ q‎‎ $‎ adjacent matrix of ‎$ G $. Then
$$ I_{G}= \langle \{ t^{\alpha_{
+}}-t^{\alpha_{-}}| \alpha \in \emph{Z}^{q} , M\alpha = 0\}\rangle =
\langle t_{1}-(x_{i}x_{j})_{1} , \cdots ,
t_{q}-(x_{i}x_{j})_{q}\rangle \cap B
$$

This ideal  ‎$ I‎_{G}‎ $ when $G$ ‎ is a bipartite graph can be characterized by combinatorics properties of ‎$ G $‎. see[3]

Let $W=\{v_{0},v_{1},\cdots , v_{r}=v_{0}\}$ be an even cycle  such that $f_{i}=x_{i-1}x_{i}$. As $$f_{1}f_{3}\cdots
f_{r-1}=f_{2}f_{4}\cdots f_{r}$$ the binomial
$$T_{W}=T_{1}T_{3}\cdots T_{r-1}-T_{2}T_{4}\cdots T_{r}$$
is in $I_{G}$. One says that $T_{W}$ is the binomial associated to
$W$.

Let $G$ be a bipartite graph then ‎$ I‎_{G}‎ $ is generated by ‎$ T‎_{W}‎ $ where ‎$ W $‎ is an even cycle. ‎In particular, if $G$ is bipartite complete graph ‎$ K‎_{2,n}‎ $‎  with $S=\{s_{1},s_{2}\}$,$T=\{t_{1},\cdots,t_{n}\}$ we have 
$$\varphi :
S=K[x_{1},\cdots,x_{n},y_{1},\cdots,y_{n}]\longrightarrow K[F]$$
induced by $\varphi(x_{i})=s_{1}t_{i} ~~~~~~~~~~~~ ,~~~~~~~~~~~~~~~~
\varphi(y_{i})=s_{2}t_{i}
 ~~~~~~~~~~~~~~~ ,~~~~~~~~~~  1\leq i \leq n$,
and  we have $I_{K_{2,n}}=\langle
x_{i}y_{j}-x_{j}y_{i}$,$1\leq i,j \leq n , i\neq j\rangle$.
 So for every connected simple graph ‎$ G $‎ on ‎$ [n] $‎,‎$ J‎_{G} ‎\subset‎ I‎_{K‎_{2,n}‎} ‎‎‎‎ $‎ and ‎$ J‎_{K‎_{n}‎}=I‎_{K‎_{2,n}‎} $‎

By a lattice we consider a finitely generated subgroup of $Z^{n}$. A partial character
$\rho$ is a homomorphism, noted by $\rho$   too, from a sub lattice  $ L_{\rho}$ to the
multiplicative group $K^{\ast} = K -\{0\}$.

 For a
partial character $\rho$ we define ‎$ I(\rho) $‎   the  Laurent binomial ideal in $K[X^{\pm}]=K[x_{1},\cdots,x_{n},x_{1}^{-1}\cdots,x_{n}^{-1}]$, generated by 
$$I(\rho)=\langle x^{m}-\rho(m) : m\in L_{\rho}\rangle$$

We let $m_{+} , m_{-} \in \emph{Z}_{+}^{n}$ denote the positive part
and negative part of a vector $m\in \emph{Z}^{n}.$  For a partial
character $\rho$, we define the ideal
$$I_{+}(\rho)=\langle x^{m_{+}}-\rho(m)x^{m_{-}} : m\in
L_{\rho}\rangle in ~~~~~~R=K[x_{1},\cdots,x_{n}]$$ Note that
$I_{+}(\rho)=I(\rho) \cap R$ and so $x^{m}-\rho(m)\in L_{\rho} $ if
and only if $x^{m_{+}}-\rho(m)x^{m_{-}} \in I_{+}(\rho).$ 

We call an ideal $I$ in $R$ a lattice ideal if there exist a partial character $\rho$ such that $I= I_{+}(\rho)$

\begin{proposition} (Eisenbud and Sturmfels [2])
If $I$ is a binomial ideal in $R=K[x_{1},\cdots,x_{n}]$ not
containing any monomial, then there is a unique partial character
$\rho$  such that $$(I : \langle x_{1},\cdots,x_{n}\rangle
^{\infty}) = I_{+}(\rho)$$
\end{proposition}

\begin{corollary}
If $I$ is a binomial ideal in $K[X]$  not
containing any monomial and $\rho$ is a partial
character then $(I : \langle x_{1},\cdots,x_{n}\rangle ^{\infty}) =
I(\rho)$ if and only if $I=I_{+}(\rho).$
\end{corollary}

In the other words, an ideal $I$ is a lattice ideal if and only if there exist a partial character $\rho$ such that $(I : \langle x_{1},\cdots,x_{n}\rangle
^{\infty}) = I_{+}(\rho)$

\begin{corollary}
The binomial edge ideal $J‎_{G}‎$ is a lattice ideal if and only if $(J‎_{G} : \langle X,Y\rangle
^{\infty}) = J‎_{G}$

\end{corollary}

 \section{\textbf{Main results}}

\begin{theorem}
The following conditions are equivalent:

i) ‎$ G $ is complete around all vertices of‎ $  G $.‎‎

ii) The ideal $J‎_{G}‎$ is a lattice ideal.
\end{theorem}
\begin{proof}
Let $J_{G}$ be a lattice ideal then by corollary 2.4 $(J_{G} : \langle X,Y \rangle
^{\infty})= J_{G}$. Now let $k\in V(G)$ and ‎$‎‎ i,j‎‎ \in N(k) $‎ so ‎$ ‎‎\lbrace‎i,k‎\rbrace‎ , ‎\lbrace‎j,k‎\rbrace‎‎ \in E(G)‎‎ $‎. We have to show that ‎$ ‎\lbrace‎i,j‎\rbrace‎ \in E(G) $‎ and for this we see that  $x_{i}y_{j}-x_{j}y_{i} \in J_{G}$. since  $x_{i}y_{k}-x_{k}y_{i} , x_{j}y_{k}-x_{k}y_{j} \in J_{G}$ and $y_{k}(y_{i}x_{j}-y_{j}x_{i})=y_{j}(x_{k}y_{i}-x_{i}y_{k}) -
y_{i}(x_{k}y_{j}-x_{j}y_{k})$ so ‎$ x_{i}y_{j}-x_{j}y_{i} \in (J_{G} : \langle X,Y \rangle
^{\infty})  $ and ‎$  x_{i}y_{j}-x_{j}y_{i} \in  J_{G} $‎ ‎ this implies that ‎$ N(K) $‎ is complete for each ‎$ k \in V(G) $‎. Conversely,
in a simple connected graph $G$, the ideal $J‎_{G} $ is a  lattice ideal [1] and the sum of two lattice ideal is again a lattice[2] so the proof is completed  by proposition 2.1.
\end{proof}

\begin{theorem}
Let ‎$ G‎_{1} , \cdots , G‎_{r}‎‎ $‎ be the connected components of $G$ and $‎\vert V(G‎_{i})‎\vert =n‎_{i}‎‎‎‎$ then the following conditions are equivalent:

i) $J‎_{G}‎$ is a lattice ideal.

ii) $J‎_{G}‎= ‎\left\langle  J‎_{G‎_{1}‎}, \cdots , J‎_{G‎_{r}‎}‎\right\rangle  = ‎\sum I‎_{K‎_{2,n‎_{i}‎}‎}‎‎‎‎‎‎$  and ‎$ J‎_{G‎_{i}‎}=I‎_{K‎_{2,n‎_{i}‎}‎}‎‎ $‎
\end{theorem}
\begin{proof}
Let  $J‎_{G}‎$ be a lattice ideal so by proposition 3.1 ‎$ G $ is complete around all vertices of‎ $  G $. ‎‎In other word each connected component of ‎$ G $ is complete. This part of proof is completed by proposition 2.1. Conversely, let  $J‎_{G}‎= ‎\left\langle  J‎_{G‎_{1}‎}, \cdots , J‎_{G‎_{r}‎}‎\right\rangle $ and ‎$ J‎_{G‎_{i}‎}=I‎_{K‎_{2,n‎_{i}‎}‎}‎‎ $‎ that means each connected component of $  G $ is complete. So 
‎$ G $ is complete around all vertices of‎ $  G $ and by proposition 3.1, $J‎_{G}‎$ is lattice.
\end{proof} 

\begin{corollary}
The following conditions are equivalent:

i) ‎$ G $ is complete around all vertices of‎ $  G $.‎‎

ii) All connected components of $G$ are complete.

iii) $J‎_{G}‎$ is a prime ideal.

iv) $J‎_{G}‎$ is a lattice ideal.

v) $J‎_{G}‎$ is a toric ideal as a sum of toric ideals associated to some bipartite complete graphs.
\end{corollary}
%---------------------------------------------------------------------------------------%

\section*{\textbf{Acknowledgement}}

The authors want to thank professor Jurgen Herzog for his useful
advice.

%---------------------------------------------------------------------------------------%

\end{document}